\RequirePackage[l2tabu, orthodox]{nag}
\documentclass[leqno,12pt]{amsart}%
\usepackage{enumerate}
\usepackage{a4wide}
\usepackage[english]{babel}
\usepackage{bbm}
\usepackage{eucal}
\usepackage{amsmath,amssymb,amsthm}
\usepackage{color}
\usepackage[colorlinks=true,linkcolor=blue]{hyperref}
\usepackage[utf8]{inputenc}
\usepackage[T1]{fontenc}

\numberwithin{equation}{section}
\theoremstyle{plain}
\newtheorem{theorem}{Theorem}[section]
\newtheorem{corollary}[theorem]{Corollary}
\newtheorem{lemma}[theorem]{Lemma}
\newtheorem{proposition}[theorem]{Proposition}

\theoremstyle{definition}
\newtheorem{definition}[theorem]{Definition}

\theoremstyle{remark}
\newtheorem{remark}[theorem]{Remark}
\newtheorem{af}[theorem]{Abstract Framework}
\newtheorem{ma}[theorem]{Main Assumptions}
\newtheorem{ada}[theorem]{Additional Assumption}

\newcommand{\sA}{\mathcal{A}}

\newcommand{\sL}{\mathcal{L}}
\newcommand{\sB}{\mathcal{B}}
\newcommand{\sM}{\mathcal{M}}
\newcommand{\sR}{\mathcal{R}}
\newcommand{\sT}{\mathcal{T}}

\newcommand{\sTtt}{(\mathcal{T}(t))_{t\ge0}}
\newcommand{\sX}{\mathcal{X}}

\newcommand{\rL}{\mathrm{L}}
\newcommand{\rM}{\mathrm{M}}

\newcommand{\rC}{\mathrm{C}}
\newcommand{\rW}{\mathrm{W}}

\newcommand{\la}{\lambda}
\renewcommand{\epsilon}{\varepsilon}
\newcommand{\ep}{\varepsilon}

\newcommand{\Am}{A_m}

\newcommand{\Ame}{A_{-1}}

\newcommand{\sBtBC}{\sB_{t}^{\textrm{BC}}}
\newcommand{\sBt}{\sB_t^{\textrm{BC}}}

\newcommand{\sBnt}{\sB_{nt}^{\textrm{BC}}}
\newcommand{\sBl}{\sB_l}

\newcommand{\sMl}{\sM_\lambda}

\newcommand{\Bl}{B_\lambda}
\newcommand{\Bm}{B_\mu}
\newcommand{\Bn}{B_\nu}
\newcommand{\LA}{B_A}
\newcommand{\Ql}{Q_\lambda}
\newcommand{\Qm}{Q_\mu}

\newcommand{\RlA}{R(\lambda,A)}
\newcommand{\RelBB}{R(\el,\BB)}
\newcommand{\RlelAA}{R(\lambda\el,\AAA)}
\newcommand{\RmA}{R(\mu,A)}
\newcommand{\RnA}{R(\nu,A)}

\newcommand{\Tme}{T_{-1}}
\newcommand{\Ttt}{(T(t))_{t\ge0}}

\newcommand{\Id}{I}

\newcommand{\Fim}{\Phi^{-}}
\newcommand{\Fiwm}{\Phi_w^{-}}

\newcommand{\Fiwp}{\Phi_w^{+}}
\newcommand{\Fwmt}{(\Fiwm)^\top}

\newcommand{\DAm}{[D(\Am)]}

\newcommand{\dX}{\partial X}
\newcommand{\FeA}{{F_1^A}}

\newcommand{\Xme}{X_{-1}}
\newcommand{\MmC}{\rM_{m}(\CC)}
\newcommand{\MnC}{\rM_{n}(\CC)}
\newcommand{\MmnC}{\rM_{m\times n}(\CC)}
\newcommand{\MnmC}{\rM_{n\times m}(\CC)}
\newcommand{\Cm}{\CC^m}
\newcommand{\Cn}{\CC^n}

\newcommand{\sRBC}{e\sR^{\textrm{BC}}}
\newcommand{\sRBCt}{e\sR_t^{\textrm{BC}}}
\newcommand{\saRBC}{a\sR^{\textrm{BC}}}
\newcommand{\saRBCt}{a\sR_t^{\textrm{BC}}}
\newcommand{\sRBCp}{e^+\sR^{\textrm{BC}}}
\newcommand{\saRBCp}{a^+\sR^{\textrm{BC}}}
\newcommand{\saRBCtp}{a^+\sR_t^{\textrm{BC}}}
\newcommand{\sRBCtp}{e^+\sR_t^{\textrm{BC}}}
\newcommand{\sRBCm}{\sR^{\textrm{BC}}_{\textrm{max}}}

\newcommand{\sRBCnt}{e\sR_{nt}^{\textrm{BC}}}
\newcommand{\sRBCntp}{e^+\sR_{nt}^{\textrm{BC}}}

\newcommand{\WepneCm}{\rW^{1,p}\bigl([0,1],\Cm\bigr)}

\newcommand{\Wkpne}{\rW^{k,p}[0,1]}

\newcommand{\Lpnl}{{\rL^p[0,l]}}
\newcommand{\Lpne}{{\rL^p[0,1]}}
\newcommand{\LpneRmp}{{\rL^p\bigl([0,1],\RR_+^m\bigr)}}

\newcommand{\Lpnep}{{\rL^p\bigl([0,1],\RR_+\bigr)}}
\newcommand{\LpntU}{{\rL^p\bigl([0,t],U\bigr)}}
\newcommand{\LpntUp}{{\rL^p\bigl([0,t],U^+\bigr)}}
\newcommand{\Lpnt}{{\rL^p[0,t]}}

\newcommand{\LpnntU}{{\rL^p\bigl([0,nt],U\bigr)}}

\newcommand{\LpneCm}{{\rL^p\bigl([0,1],\Cm\bigr)}}

\newcommand{\LelRpU}{\rL^1_{\mathrm{loc}}(\RR_+,U)}

\newcommand{\LentU}{\rL^1([0,t],U)}

\newcommand{\WzentU}{\rW^{2,1}([0,t],U)}

\newcommand{\CeRpX}{\rC^1(\RR_+,X)}

\newcommand{\AAA}{\mathbb A}
\newcommand{\BB}{\mathbb B}
\newcommand{\CC}{\mathbb C}
\newcommand{\NN}{\mathbb N}
\newcommand{\RR}{\mathbb R}

\newcommand{\eins}{\mathbbm{1}}
\newcommand{\eab}{\eins_{[\alpha,\beta]}}
\newcommand{\ela}{e^{\lambda\alpha}}
\newcommand{\elb}{e^{\lambda\beta}}
\newcommand{\els}{e^{\lambda s}}
\newcommand{\elt}{e^{\lambda t}}
\newcommand{\el}{e^{\lambda}}

\newcommand{\Rg}{\operatorname{rg}}

\newcommand{\lin}{\operatorname{span}}
\newcommand{\linq}{\overline{\operatorname{span}}\,}
\newcommand{\gb}{\operatorname{{\omega_0}}}

\newcommand{\coc}{\overline{\operatorname{co}}\,}
\newcommand{\co}{\operatorname{co}}

\newcommand{\SBC}{$\Sigma_{\textrm{BC}}(\Am,B,Q)$}

\newcommand{\ds}{\,ds}
\newcommand{\dr}{\,dr}

\newcommand{\dds}{\frac{d}{ds}}

\renewcommand{\l}{\left}
\renewcommand{\r}{\right}

\newcommand{\diag}{\operatorname{diag}}
\newcommand{\epsl}{\varepsilon_{\la}}

\renewcommand{\phi}{\varphi}

\newcommand{\etalchar}[1]{$^{#1}$}

\advance\textheight by 24mm
\title[Exact and Positive Controllability]
{Exact and Positive Controllability of Boundary Control Systems}
\author{Klaus-Jochen Engel and Marjeta Kramar Fijav\v{z}}
\date{\today}

\begin{document}

\subjclass[2010]{93B05, 47N70, 35R02}%
\keywords{Boundary control, positive control, operator semigroups, reachability spaces, Kalman condition, flows in networks.}%

\begin{abstract}
Using the semigroup approach to abstract boundary control problems  we characterize the space of all \emph{exactly} reachable states.
Moreover, we study the situation when the controls of the system are required to be \emph{positive}. The abstract results are applied to flows in networks with static as well as dynamic boundary conditions. 
\end{abstract}
\maketitle

\section{Introduction}
This paper is a continuation of \cite{EKNS:08, EKKNS:10} where we introduced a semigroup approach to boundary control problems and applied it to the control of flows in networks. While in these previous works we concentrated on maximal \emph{approximate} controllability, we now focus on the \emph{exact}- and \emph{positive}  controllability spaces. In particular, this will generalize and refine results given in \cite{BBEAM:13, EKNS:08, EKKNS:10} where further references to the related literature can be found.

\smallbreak
As a simple motivation, we consider as in \cite{EKNS:08} a transport process along the edges of a finite network. This system is governed through the transmission conditions in the vertices of the network which represent the ``\emph{boundary space}'' for our problem. We then like to control the behavior of this system by acting upon a single node only. In this context it is reasonable to ask the following questions.

\begin{itemize}
\item Can we reach all possible states in final time?
\item If not, can we describe the maximal possible set of reachable states?
\item Is the choice of a particular control node important?
\item Which states can be reached if only positive controls are allowed?
\end{itemize}

In Section~\ref{examples} we will address all these questions. To this end we first recall in Section~\ref{TAF} our abstract framework from \cite{EKKNS:10} as well as some basic results concerning boundary control systems. In Section~\ref{sec:EC} we then characterize boundary admissible control operators and describe the corresponding exact reachability space. In Section~\ref{Sec:pos} we turn our attention to positive boundary control systems on Banach lattices. Finally, in Section~\ref{examples} we apply our results and explicitly compute the exact (positive)  reachability spaces for three different examples of a transport equation controlled at the boundary: in $\RR^m$, in a network, and in a network with dynamic boundary conditions.

\section{The Abstract Framework}\label{TAF}
We start by recalling our general setting from \cite{EKKNS:10}.

\begin{af}\label{af-bcs}
We consider
\begin{enumerate}[(i)]
\item three Banach spaces $X$, $\dX$ and $U$, called the
\emph{state}, \emph{boundary} and \emph{control space}, resp.;
\item a closed, densely defined \emph{system operator} $\Am:D(\Am)\subseteq X\to
X$;
\item a \emph{boundary operator} $Q\in\sL(\DAm,\dX)$;
\item a \emph{control operator} $B\in\sL(U,\dX)$.
\end{enumerate}
\end{af}

For these operators and spaces and a \emph{control function}
$u\in\LelRpU$ we then consider the \emph{abstract Cauchy problem
with boundary control}\footnote{We denote by $\dot x(t)$ the derivative of $x$ with respect to the ``time'' variable $t$.}
\begin{alignat}{2}\label{ACPBC}
\begin{cases}
\dot x(t)=\Am x(t),&t\ge0,\\
Qx(t)=Bu(t), &t\ge0,\\
x(0)=x_0.
\end{cases}
\end{alignat}
A function $x(\cdot)=x(\cdot,x_0,u)\in\CeRpX$ with $x(t)\in D(\Am)$
for all $t\ge0$ satisfying \eqref{ACPBC} is called a \emph{classical
solution}. Moreover, we denote the \emph{abstract boundary control
system} associated to \eqref{ACPBC} by \SBC{}.

\smallbreak In order to investigate \eqref{ACPBC} we make the following standing assumptions which in particular ensure that the \emph{un}controlled abstract Cauchy problem, i.e., \eqref{ACPBC} with $B=0$, is well-posed.

\begin{ma}
\makeatletter
\hyper@anchor{\@currentHref}%
\makeatother
\label{ma-bcs} 
\begin{enumerate}[(i)]
\item The restricted operator $A\subset\Am$ with domain $D(A):=\ker Q$ generates
a strongly continuous semigroup $\Ttt$ on $X$;
\item the boundary
operator $Q:D(\Am)\to\dX$ is surjective.
\end{enumerate}
\end{ma}

Under these assumptions the following properties have been shown in \cite[Lem.~1.2]{Gre:87}.

\begin{lemma}\label{lem-Gre} Let Assumptions~\ref{ma-bcs} be
satisfied. Then the following assertions are true for all $\lambda,\mu\in\rho(A)$.
\begin{enumerate}[(i)]
\item $D(\Am)=D(A)\oplus\ker(\lambda-\Am)$;
\item $Q|_{\ker(\lambda-\Am)}$ is invertible and the operator
\[\Ql:=(Q|_{\ker(\lambda-\Am)})^{-1}:\dX\to\ker(\lambda-\Am)\subseteq X\]
is bounded;
\item $\RmA\Ql= 
\RlA\Qm$. 
\end{enumerate}
\end{lemma}

The following operators are essential to obtain explicit
representations of the solutions of the boundary control problem \eqref{ACPBC}.

\begin{definition}
For $\lambda\in\rho(A)$ we call the operator $\Ql$, introduced in
Lemma~\ref{lem-Gre}.(ii), abstract \emph{Dirichlet operator} and define
\[\Bl:=\Ql B \in \sL\bigl(U,\ker(\lambda-\Am)\bigr)\subset\sL(U,X).\]
\end{definition}

By \cite[Prop.~2.7]{EKKNS:10} the solutions of \eqref{ACPBC} can be represented by the following extrapolated version of the variation of parameters formula.

\begin{proposition}\label{prop-vpf-ex} Let $x_0\in X$, $u\in\LelRpU$ and $\lambda\in\rho(A)$. If
$x(\cdot)=x(\cdot,x_0,u)$ is a classical solution \eqref{ACPBC}, then it
is given by the \emph{variation of parameters formula}
\begin{equation}\label{eq-vpf-1}
x(t)=T(t)x_0+(\lambda-\Ame)\int_0^t T(t-s)\Bl u(s)\ds,\quad
t\ge0.
\end{equation}
\end{proposition}

Our aim in the sequel is to investigate which states in $X$ can be \emph{exactly} reached from $x_0=0$ by solutions of \eqref{ACPBC}. To this end we have to impose an additional assumption which, by \eqref{eq-vpf-1}, ensures that solutions for $\rL^p$-controls have values in $X$.

\begin{definition}
Let $1\le p\le+\infty$.
Then the control operator $B\in\sL(U,\dX)$ is called \emph{$p$-boundary admissible} if there exist $t>0$  and $\lambda\in\rho(A)$ such that
\begin{equation}\label{b-admiss}
\int_0^{t} T(t-s)\Bl u(s)\ds\in D(A)\quad\text{for all }u\in\LpntU.
\end{equation}
\end{definition}

\begin{remark}\label{rem:Bdd-op}
From Lemma~\ref{lem-Gre}.(iii) it follows that $(\lambda-\Ame)Q_\lambda\in\sL(\dX,\Xme)$,  hence also 
\[
\LA:=(\lambda-\Ame)\Bl\in\sL(U,\Xme)
\]  
is independent of $\lambda\in\rho(A)$. Then $B\in\sL(U,\dX)$ is $p$-boundary admissible if and only if $\LA$ is $p$-admissible in the usual sense, cf. \cite[Def.~4.1]{Wei:89a}.
This implies that if \eqref{b-admiss} is satisfied for some $t>0$ then it is satisfied for every $t>0$. Moreover, we note that $B$ is $1$-boundary admissible if $\ker(\lambda-\Am)\subset\FeA$, see \cite[Lem.~A.3]{EKKNS:10}. Finally, since $\LpntU\subset\LentU$ it follows that $1$-boundary admissibility implies $p$-boundary admissibility for all $p>1$.
\end{remark}

\smallbreak
Now assume that $B\in\sL(U,\dX)$ is $p$-boundary admissible. Then for fixed $\lambda\in\rho(A)$ and $t>0$ the operators $\sBtBC:\LpntU\to X$ given by
\begin{equation}\label{bt-bc}
\sBtBC u:= (\lambda- A)\int_0^tT(t-s)\Bl u(s)\ds\\
=\int_0^t\Tme(t-s)\LA u(s)\ds
\end{equation}
are called the \emph{controllability maps} of the system \SBC, where the second integral initially is taken in the extrapolation space $\Xme$. Note that by the closed graph theorem $\sBtBC\in\sL(\LpntU,X)$. Hence,  this definition is independent of the particular choice of $\lambda\in\rho(A)$ and gives the (unique) classical solution of \eqref{ACPBC} for given $u\in \WzentU$ and  $x_0=0$. This motivates the following definition.

\begin{definition} 
\begin{enumerate}[(a)]
\item
The \emph{exact reachability space in time $t\ge0$} of \SBC{} is defined by\footnote{By $\Rg(T)$ we denote the range $TX\subseteq Y$ of an operator $T:X\to Y$.}
\begin{equation}\label{eq-reach-sp-t}
\sRBCt:=\Rg(\sBtBC).
\end{equation}
Moreover, we define the \emph{exact reachability space} (in arbitrary time) by
\begin{equation}\label{eq-reach-sp}
\sRBC:=\bigcup_{t\ge0}\Rg(\sBtBC)
\end{equation}
and call \SBC{} \emph{exactly controllable} (in arbitrary time) if $\sRBC=X$.
\item
The \emph{approximate reachability space in time $t\ge0$} of \SBC{} is defined by
\begin{equation}\label{eq-app-reach-sp-t}
\saRBCt:=\overline{\sRBCt}.
\end{equation}
Moreover, we define the \emph{approximate reachability space} (in arbitrary time) by
\begin{equation}\label{eq-app-reach-sp}
\saRBC:=\overline{\;\bigcup_{t\ge0}\saRBCt\,}
\end{equation}
and call \SBC{} \emph{approximately controllable} if $\saRBC=X$.
\end{enumerate}
\end{definition}

From \cite[Thm.~2.12 \& Cor.~2.13]{EKKNS:10} we obtain the following properties and representations of the approximate reachability space.

\begin{proposition}\label{prop:aR}
Assume that $B\in\sL(U,\dX)$ is $p$-boundary admissible.
Then the following holds.
\begin{enumerate}[(i)]
\item $\saRBC$ is a closed linear subspace, invariant under $\Ttt$ and $R(\lambda,A)$ for $\lambda>\gb(A)$.
\item $\saRBC=\linq \bigcup_{\lambda>\omega} \Rg(\Bl)$ for some $\omega>\gb(A)$.
\item $\saRBC\subseteq\linq \bigcup_{\lambda>\gb(A)}\ker(\lambda-\Am)$.
\end{enumerate}
\end{proposition}

Part~(iii) shows that there is an upper bound for the reachability space
depending on the eigenvectors of $\Am$ only, independent of the
control operator $B$. This justifies the following notion.

\begin{definition}\label{RBCmax}
The \emph{maximal reachability space} of \SBC{} is defined by
\begin{equation*}
\sRBCm:=\linq\bigcup_{\lambda>\gb(A)} \ker
(\lambda-A_m).
\end{equation*}
The system \SBC{} is called \emph{maximally controllable} if $\sRBC=\sRBCm$.
\end{definition}

We stress that $\sRBCm\ne X$ may happen (some basic examples are provided in \cite[Sec.~5]{EKNS:08}), hence the relevant question about exact or approximate
controllability is indeed to compare $\sRBC$ or $\saRBC$ to the space
$\sRBCm$ and not to the whole space $X$, as it is usually
done in the classical situation.

After this short summary on boundary control systems  \SBC{} taken mainly from \cite{EKKNS:10} in the context of approximate controllability, we now turn our attention to the case of \emph{exact} controllability. 

\section{Exact controllability}
\label{sec:EC}

%
%

\smallbreak

We start this section by giving two characterizations of $p$-boundary admissibility for a control operator $B$ which frequently simplifies the explicit computation of the associated controllability map $\sBt$.
Here for $\lambda\in\CC$ we introduce the function $\epsl:\RR\to\CC$ by $\epsl(s):=\els$.
Moreover, for $f\in\Lpnt$ and $u\in U$ we define
\[
f\otimes u\in\LpntU
\quad\text{by}\quad
(f\otimes u)(s):=f(s)\cdot u.
\]
Finally, we denote by $\eab$ the characteristic function of the interval $[\alpha,\beta]\subset[0,t]$.

\begin{proposition}\label{prop:range-B1}
For a control operator $B\in\sL(U,\dX)$ the following are equivalent.
\begin{enumerate}[(a)]
\item $B$ is $p$-boundary admissible.
\item There exist $\lambda\in\rho(A)$, $t>0$ and $M\in\sL\bigl(\LpntU,X\bigr)$ such that for all $0\le\alpha\le\beta\le t$ and $v\in U$
\begin{equation}\label{eq:strange1}
\bigl(\elb\,T(t-\beta)-\ela\,T(t-\alpha)\bigr)\Bl v=
M(\epsl\cdot\eab\otimes v).
\end{equation}
\item There exist $t>0$, $\lambda_0>\gb(A)$ and $M\in\sL\bigl(\LpntU,X\bigr)$ such that for all $\lambda\ge\lambda_0$ and $v\in U$
\begin{equation}\label{eq:strange2}
\bigl(\elt-T(t)\bigr)\Bl v=
M(\epsl\otimes v).
\end{equation}
\end{enumerate}
Moreover, in this case the controllability map is given by $\sBt=M$.
\end{proposition}



\begin{proof}  Let $u=\epsl\cdot\eab\otimes v$ for some $\lambda\in\rho(A)$, $0\le\alpha\le\beta\le t$ and $v\in U$. Then
\begin{align}
\int_0^tT(t-s)\Bl u(s)\ds\notag
&=\elt\int_\alpha^\beta e^{-\lambda(t-s)}\, T(t-s)\Bl v\ds\\\notag
&=\elt\int_{t-\beta}^{t-\alpha} e^{-\lambda s}\, T(s)\Bl v\ds\\\label{eq:B-zul0}
&=\RlA\cdot\bigl(\elb\,T(t-\beta)-\ela\,T(t-\alpha)\bigr)\Bl v.
\end{align}

(a)$\Rightarrow$(b). Since by assumption $B$ is $p$-boundary admissible we have $\sBtBC\in\sL(\LpntU,X)$. Hence, \eqref{bt-bc} and \eqref{eq:B-zul0} imply \eqref{eq:strange1} for $M=\sBt$.

\smallbreak
(b)$\Rightarrow$(a). We start by proving \eqref{b-admiss}. The idea is to show this first for functions of the type $u=\epsl\cdot\eab\otimes v$. Then by linearity it also holds for linear combinations of such functions and a density argument implies \eqref{b-admiss} for arbitrary $u\in\LpntU$. To this end let $u=\epsl\cdot\eab\otimes v$ for $[\alpha,\beta]\subset[0,t]$ and $v\in U$. Then \eqref{eq:strange1} and \eqref{eq:B-zul0} imply
\begin{align}
\int_0^tT(t-s)\Bl u(s)\ds\notag
&=\RlA\cdot M(\epsl\cdot\eab\otimes v)\\
&=\RlA\cdot Mu\label{eq:B-zul}.
\end{align}
Note that the multiplication operator $\sMl\in\sL\bigl(\LpntU\bigr)$ defined by $\sMl u:=\epsl\cdot u$ is an isomorphism (with bounded inverse $\sM_{-\lambda}$). Hence, it maps dense sets of $\LpntU$ into dense sets. Since the step functions are dense in $\LpntU$ (see \cite[p.14]{ABHN:01}), the linear combinations of functions of the type $\epsl\cdot\eab\otimes v$ for $[\alpha,\beta]\subset[0,t]$ and $v\in U$ form a dense subspace of $\LpntU$. Thus, we conclude that \eqref{eq:B-zul} holds for all $u\in\LpntU$. Clearly this implies that $B$ is $p$-boundary admissible and $\sBt=M$.

\smallbreak
Recall that $\LA=(\lambda-\Ame)\Bl$ is independent of $\lambda\in\rho(A)$. Hence, the equivalence (a)$\Leftrightarrow$(c) follows by similar arguments replacing the total set $\{\epsl\cdot\eab\otimes v:0\le\alpha<\beta\le t, v\in U\}$ by the set $\{\epsl\otimes v:\lambda\ge\lambda_0, v\in U\}$ which by the Stone--Weierstra{\ss} theorem is total as well in $\LpntU$ for all $\lambda_0>\gb(A)$.
\end{proof}

We note that by linearity it would suffice that Part~(b) of Proposition~\ref{prop:range-B1} is satisfied for $\alpha=0$ and all $0\le\beta\le t$ (or for all $0\le\alpha\le t$ and $\beta=t$).

\begin{corollary}\label{cor:Bn}
Let\footnote{We use the notation $\NN_l:=\{l,l+1,l+2,\ldots\}$ for the set of natural numbers starting at $l\in\NN$.} $n\in\NN_1$ and assume that $B$ is $p$-boundary admissible.
Then for all $u\in\LpnntU$
\begin{equation}\label{eq:sBnt}
\sBnt u=\sum_{k=0}^{n-1}T(t)^k M u_{k}
\end{equation}
where $u_k\in\LpntU$ is defined by
\begin{equation}\label{eq:def-uk}
u_{k}(s)=u\bigl((n-k-1)t+s\bigr)
\end{equation}
and $M\in\sL\bigl(\LpntU,X\bigr)$ is the operator from Proposition~\ref{prop:range-B1}. 
\end{corollary}

\begin{proof} Let $u\in\LpnntU$. Then  by \eqref{bt-bc}
\begin{align*}
\sBnt u
&=(\lambda-A)\int_0^{nt}T(nt-s)\Bl\, u(s)\ds\\
&=(\lambda-A)\sum_{k=1}^n T\bigl((n-k)t\bigr)\int_{(k-1)t}^{kt} T(kt-s)\Bl\, u(s)\ds\\
&=\sum_{k=1}^n T\bigl((n-k)t\bigr)\cdot(\lambda-A)\int_{0}^t T(t-s)\Bl\, u_{n-k}(s)\ds\\
&=\sum_{k=0}^{n-1} T(t)^k\, \sBt u_{k}.
\qedhere
\end{align*}
\end{proof}

In Section \ref{examples} we will see that \eqref{eq:strange1}, \eqref{eq:strange2}, and  \eqref{eq:sBnt} allow us to easily compute the controllability map in the situations studied in \cite[Sect.~4]{EKNS:08} and \cite[Sect.~3]{EKKNS:10} dealing with the control of flows in networks.

\begin{corollary}\label{cor:Reach}
If $B$ is $p$-boundary admissible,
 then the exact reachability space in time $nt$ for $n\in\NN_1$ is given by
\begin{equation*}
\sRBCnt=\l\{\sum_{k=0}^{n-1}T(t)^k M u_k : u_k\in\LpntU,\; 1\le k\le n-1 \r\},
\end{equation*}
where $M\in\sL\bigl(\LpntU,X\bigr)$ is the operator from Proposition~\ref{prop:range-B1}.
\end{corollary}

\section{Positive controllability}\label{Sec:pos}

In this section we are interested in positive control functions yielding positive states. To this end we will make the following

\begin{ada}
The spaces
$X$ 
and $U$ are Banach lattices.
\end{ada}
Moreover, by $Y^+:=\{y\in Y:Y\ge0\}$ we denote the positive cone in a Banach lattice $Y$.

\smallbreak
Note that in the sequel we do \emph{not} make any positivity assumptions on $\Ttt$, $B$ or $\Ql$ if not stated otherwise.

\begin{definition}
\begin{enumerate}[(a)]
\item The \emph{exact positive reachability space in time $t\ge0$} of system \SBC{} is defined by
\begin{equation}\label{eq-reach-sp-t+}
\sRBCtp:=\Bigl\{\sBtBC u:u\in\LpntUp\Bigr\}.
\end{equation}
Moreover, we define the \emph{exact positive reachability space} (in arbitrary time) by
\begin{equation}\label{eq-reach-sp+}
\sRBCp:=\bigcup_{t\ge0}\sRBCtp
\end{equation}
and call \SBC{}  \emph{exactly positive controllable} (in arbitrary time) if \\$\sRBCp=X^+$.
\smallbreak
\item The \emph{approximate positive reachability space in time $t\ge0$}  of \SBC{} is defined by
\begin{equation}\label{eq-a-reach-sp+}
\saRBCtp:=\overline{\sRBCtp}.
\end{equation}
Moreover, we define the \emph{approximate positive reachability space} (in arbitrary time) by
\begin{equation}\label{eq-app-reach-sp+}
\saRBCp:=\overline{\;\bigcup_{t\ge0}\saRBCtp\,}
\end{equation}
and call \SBC{}  \emph{approximately positive controllable} if $\saRBCp=X^+$.
\end{enumerate}
\end{definition}

%
%
%
%
%
First we give necessary and sufficient conditions implying that starting from the initial state $x_0=0$ positive controls result in positive states.

\begin{proposition}\label{prop:reach-pos}
 Assume that $B\in\sL(U,\dX)$ is $p$-boundary admissible. Then
\begin{equation}\label{eq:pos-R-1}
\sRBCtp\subset X^+
\end{equation}
if and only if
\begin{equation}\label{eq:pos-R-1.5}
\saRBCtp\subset X^+
\end{equation}
if and only if there exists $\lambda\in\RR\cap\rho(A)$ such that
\begin{equation}\label{eq:pos-R-2}
\bigl(\elb\,T(t-\beta)-\ela\,T(t-\alpha)\bigr)\Bl\ge0
\qquad\text{for all $0\le\alpha\le\beta\le t$.}
\end{equation}
Moreover, if $\Ttt$ is positive, then the above assertions are satisfied if and only if 
\begin{equation}\label{eq:pos-R-1x}
\sRBCp\subset X^+
\end{equation}
if and only if
\begin{equation}\label{eq:pos-R-1.5x}
\saRBCp\subset X^+
\end{equation}
if and only if there exists $\lambda>\gb(A)$ and $t>0$ such that
\begin{equation}\label{eq:pos-R-3}
\bigl(\els-T(s)\bigr)\Bl\ge0
\qquad\text{for all $0\le s\le t$}
\end{equation}
if and only if there exists  $\lambda_0>\gb(A)$ such that
\begin{equation}\label{eq:pos-R-4}
\Bl\ge0
\qquad\text{for all $\lambda\ge\lambda_0$}.
\end{equation}
\end{proposition}


\begin{proof} The equivalence of \eqref{eq:pos-R-1} and  \eqref{eq:pos-R-1.5} follows from the closedness of $X^+$. To show the equivalence of \eqref{eq:pos-R-1} and \eqref{eq:pos-R-2} recall that by \cite[p.14]{ABHN:01} the step functions are dense in $\LpntU$. Since the map $u\mapsto u^+$ on $\LpntU$ is continuous, we conclude that  the positive step functions are dense in $\LpntUp$.
The claim then follows from (the proof of) Proposition~\ref{prop:range-B1} using the boundedness of the controllability map $\sBt$.

Now assume that $\Ttt$ is positive. Then the equivalences of \eqref{eq:pos-R-1}, \eqref{eq:pos-R-1.5}  with \eqref{eq:pos-R-1x}, \eqref{eq:pos-R-1.5x} follow from Corollary~\ref{cor:Reach} using the fact that the reachability spaces are growing in time. In particular, this implies that if \eqref{eq:pos-R-2} holds for some $t>0$ it holds for arbitrary $t>0$ and choosing $\beta=t$ and $\alpha=0$ we obtain \eqref{eq:pos-R-3} for arbitrary $t>0$.

To show the remaining assertions we fix some $\lambda>\gb(A)$ and define on $\sX:=X\times X$ the operator matrix
\[
\sA:=
\begin{pmatrix}
A-\lambda&0\\0&0
\end{pmatrix},\quad
D(\sA):=\Bigl\{\tbinom{x}{y}\in D(A_m)\times\dX:Qx=By\Bigr\}.
\]
Then by \cite[Cor.~3.4]{Eng:99} the matrix $\sA$ generates a $C_0$-semigroup $\sTtt$ given by
\begin{equation}\label{eq:sTt}
\sT(s)=\begin{pmatrix}
e^{-\lambda s}T(s)&\bigl(\Id-e^{-\lambda s}T(s)\bigr)B_{\lambda}\\0&\Id
\end{pmatrix},
\quad s\ge0.
\end{equation}
Moreover, by \cite[Lem.~3.1]{Eng:99} we have $(0,+\infty)\subset\rho(\sA)$ and
\begin{equation} \label{eq:RsA}
R(\mu,\sA)=
\begin{pmatrix}
R(\mu+\lambda,A)&\frac1\mu B_{\mu+\lambda}\\
0&\frac1\mu
\end{pmatrix}
\quad\text{for $\mu>0$}.
\end{equation}
Now, if \eqref{eq:pos-R-3} holds then $\sT(s)\ge0$ for all $0\le s\le t$ which implies that $\sTtt$ is positive which is equivalent to the fact that $\sA$ is resolvent positive. However, by \eqref{eq:RsA} the latter is the case if and only if \eqref{eq:pos-R-4} is satisfied which shows the equivalence of \eqref{eq:pos-R-3} and \eqref{eq:pos-R-4}.
Finally, if \eqref{eq:pos-R-3} holds, then
\begin{align*}
\bigl(e^{\lambda\beta}\,T(t-\beta)-e^{\lambda\alpha}\,T(t-\alpha)\bigr)\Bl
&=e^{\lambda\beta}\,T(t-\beta)\cdot\bigl(\Id-e^{-\lambda(\beta-\alpha)}\,T(\beta-\alpha)\bigr)\Bl\\
&\ge0
\end{align*}
for all $0\le\alpha\le\beta\le t$. This proves \eqref{eq:pos-R-2} and completes the proof. 
\end{proof}

In the sequel we use the notation $\co M$ and $\coc M$ to indicate the  convex hull and the closed convex hull of a set $M\subset X$, respectively.

\begin{proposition}\label{prop:aRp-char}
Assume that $B\in\sL(U,\dX)$ is $p$-boundary admissible and that $\sRBCtp\subset X^+$.
Then the following holds.
\begin{enumerate}[(i)]
\item $\saRBCp$ is a closed convex cone, invariant under $\Ttt$ and $R(\lambda,A)$ for $\lambda>\gb(A)$.
\item 
$
\saRBCp=\coc\l\{\bigl(e^{\lambda \beta}T(t-\beta)-e^{\lambda \alpha}T(t-\alpha)\bigr)\Bl v:0\le \alpha\le\beta\le t,\,v\in U^+\r\}
$
for all $\lambda>\gb(A)$.
\item $\saRBCp=\coc\{T(t)\Bl v:t\ge0,\,\lambda>w,\,v\in U^+\}$ for all $w>\gb(A)$.
\item $\saRBCp=\coc\{\RlA^n\Bl v:n\in\NN_0,\,\lambda>w,\,v\in U^+\}$ for some/all $w>\gb(A)$.
\end{enumerate}
\end{proposition}

\begin{proof} (i). Clearly, $\saRBCp$ is a closed convex cone.  Its invariance  under $\Ttt$ and $\RlA$ for $\lambda>\gb(A)$ follows from the representations in (iii) and (iv).

To show (ii) we note that by \eqref{bt-bc} and \eqref{eq:B-zul0} the inclusion ``$\supseteq$'' holds.
Now recall that  the positive step functions are dense in $\LpntUp$ and invariant under positive convey combinations. Hence, the boundedness of the controllability maps implies equality of the spaces in (ii).

\smallbreak

To obtain (iii) we note that by \eqref{bt-bc} and \eqref{eq:B-zul0} we have
\[
\bigl(\elb\,T(t-\beta)-\ela\,T(t-\alpha)\bigr)\Bl v\in\sRBCp
\]
for all $0\le\alpha\le\beta\le t$ and $v\in U^+$.
Multiplying this inclusion by $e^{-\lambda\beta}>0$ and putting $s:=t-\beta$ and $r:=t-\alpha$ implies
\[
\bigl(T(s)-e^{\lambda(s-r)}\,T(r)\bigr)\Bl v\in\sRBCp
\]
for all $0\le s\le r$ and $v\in U^+$. Since $\lambda>\gb(A)$ we obtain
\[
\lim_{r\to+\infty}e^{\lambda(s-r)}\,\|T(r)\|=0
\]
and hence
\[T(s)\Bl v\in\saRBCp\]
for all $s\ge0$ and $v\in U^+$. This shows the inclusion ``$\supseteq$'' in (iii). 
For the converse inclusion in (iii) it suffices to prove that
\begin{equation}\label{eq:inc2.(ii)}
\sRBCtp\subset\coc\bigl\{T(s)\Bm y:s\ge0,\,\mu>w,\,y\in U^+\bigr\}
\end{equation}
for all $t>0$ and $w>\gb(A)$. Since $B\in\sL(U,\dX)$ is $p$-boundary admissible the controllability map $\sBt$ is continuous. Moreover, the positive step functions are dense in $\LpntUp$ and $\coc\{T(s)\Bm y:s\ge0,\,\mu>w,\,y\in U^+\}$ is a convex cone. Combining these facts and \eqref{eq:B-zul0}  it follows that \eqref{eq:inc2.(ii)} holds if
\begin{equation}\label{eq::inc2.(ii)-1}
\bigl(\elb T(t-\beta)-\ela T(t-\alpha)\bigr)\Bl v\in \coc\bigl\{T(s)\Bm y:s\ge0,\,\mu>w,\,y\in U^+\bigr\}
\end{equation}
for all  $0\le\alpha\le\beta\le t$, $k\in\NN_0$ and $v\in X^+$. 
Since $\Ttt$ is strongly continuous the following integral is the limit of Riemann sums, hence for $\nu>\max\{0,w\}$ we obtain using Lemma~\ref{lem-Gre}.(iii)
\begin{align*}
\coc\bigl\{T(s)\Bm y:s\ge0,\,\mu>w,\,y\in U^+\bigr\}&\ni\nu\int_{t-\beta}^{t-\alpha}e^{\lambda(t-r)}T(r)\Bn v\dr\\
&=\bigl(\elb\,T(t-\beta)-\ela\,T(t-\alpha)\bigr)\nu\RlA\Bn v\\
&=\bigl(\elb\,T(t-\beta)-\ela\,T(t-\alpha)\bigr)\nu\RnA\Bl v\\
&\to\bigl(\elb\,T(t-\beta)-\ela\,T(t-\alpha)\bigr)\Bl v,
\end{align*}
as $\nu\to+\infty$. 
This proves \eqref{eq::inc2.(ii)-1} and completes the proof of (iii). 
\smallbreak
That the right-hand-sides of the equalities in (iii) and (iv) coincide follows from the integral representation of the resolvent (see \cite[Cor.~II.1.11]{EN:00}) and the Post--Widder inversion formula (see \cite[Cor.~III.5.5]{EN:00}). For the details we refer to the proof of \cite[Prop.~3.3]{BBEAM:13}.
\end{proof}

%

\begin{corollary}
Assume that $B\in\sL(U,\dX)$ is $p$-boundary admissible and that $\saRBCp\subset X^+$. Then the following are equivalent.
\begin{enumerate}[(a)]
\item The system \SBC{} is approximately positive controllable.
\item There exists $w>\gb(A)$ such that  the following implication holds for all $\phi\in X'$
\[
\big<T(s)\Bl v,\phi\big>\ge0\text{ for all $v\in U^+$, $s\ge0$ and $\lambda>w$ $\Rightarrow$ $\phi\ge0$.}
\]
\item There exists $w>\gb(A)$ such that the following implication holds for all $\phi\in X'$
\[
\bigl<\RlA^n\Bl v,\phi\bigr>\ge0 \text{ for all $v\in U^+$, $n\in\NN$ and $\lambda>w$ $\Rightarrow$ $\phi\ge0$.}
\]
\end{enumerate}
\end{corollary}

\begin{proof}
This follows from the proof of \cite[Thm.~3.4]{BBEAM:13} by replacing \cite[Prop.~3.3]{BBEAM:13}  with our Proposition~\ref{prop:aRp-char}.
\end{proof}

\begin{remark}
The previous two results generalize \cite[Prop.~3.3 and Thm.~3.4]{BBEAM:13}, respectively, where it is assumed that $\Ttt$, $B$ and $Q_{\lambda}$ for all $\lambda>\lambda_0$ are all positive and, in particular, the additional hypothesis
\begin{itemize}
\item[(H)]  \emph{There exists $\gamma>0$ and $\lambda_0\in\RR$ such that $\|Qx\|\ge\gamma\lambda\|x\|$ for all $\lambda>\lambda_0$ and $x\in\ker(\lambda-A_m)$}
\end{itemize}
is made. We note that Hypothesis~(H) in reflexive state spaces $X$ implies that $A=A_m$, cf. \cite[Lem.~A.1]{ABE:16}. Hence, the results of \cite{BBEAM:13} are, e.g., not applicable to  state space like $X=\rL^p([a,b],Y)$ for $p\in(1,+\infty)$ and reflexive $Y$.
\end{remark}

\smallskip
Combining Corollary~\ref{cor:Bn} and Proposition \ref{prop:reach-pos}
we finally obtain the following characterization of an exact positive reachability space.

\begin{corollary}\label{cor:Reach_pos}
Assume that $B$ is $p$-boundary admissible, $t>0$ and $n\in\NN_1$.
Then the exact positive reachability space in time $nt$ is given by
\begin{equation*}
\sRBCntp=\l\{\sum_{k=0}^{n-1}T(t)^k M u_k : u_k\in\LpntUp,\; 1\le k\le n-1 \r\},
\end{equation*}
where $M\in\sL\bigl(\LpntUp,X\bigr)$ is the operator from Proposition~\ref{prop:range-B1}. Moreover, the operator $M$ is positive  if and only if $\saRBCtp\subset X^+$.
\end{corollary}

\section{Examples}\label{examples}

In this section we will show how our abstract results can be applied to a transport equation with boundary control and to the vertex control of flows in networks.

\subsection{Exact Boundary Controllability for a Transport Equation}

In this subsection we study a transport equation in $\RR^m$ given by\footnote{We denote by $x'(t,s)$ the derivative of $x(t,s)$ with respect to the ``space'' variable $s$.}
\begin{alignat}{2}\label{eq:TE}
\begin{cases}
\dot x(t,s)=x'(t,s),& s\in[0,1],\ t\ge0,\\
x(t,1)=\BB x(t,0)+u(t)\cdot b,&t\ge0,\\
x(0,s)=0,& s\in[0,1].
\end{cases}
\end{alignat}
Here $x:\RR_+\times[0,1]\to\CC^m$, $\BB\in\MmC$, $u:\RR_+\to\CC$ and $b\in\Cm$. In order to fit this system in our general framework we choose

\begin{itemize}
\item the  state space $X:=\LpneCm$ where $1\le p<+\infty$, 
\item the boundary space $\dX:=\Cm$,
\item the control space $U:=\CC$,
\item the control operator $B:=b\in\Cm\simeq\sL(U,\dX)=\sL(\CC,\Cm)$,
\item the system operator 
\[\Am:=\diag\l(\dds\r)_{m\times m}\quad\text{with domain}\quad D(\Am):=\WepneCm,\]
\item the boundary operator $Q:\WepneCm\to\Cm$, $Qf:=f(1)-\BB f(0)$,
\item the operator $A\subset\Am$ with domain $D(A)=\ker Q$,
\item the state trajectory $x:\RR_+\to\LpneCm$, $x(t):=x(t,\cdot)$.
\end{itemize}
With these choices the controlled transport equation~\eqref{eq:TE} can be reformulated as an abstract Cauchy problem with boundary control of the form \eqref{ACPBC}. Clearly, the boundary operator $Q$ is surjective. 

\smallbreak
By \cite[Cor.~18.4]{BKR:17}
we know that for $\lambda\in\CC$ and $A=\Am|_{\ker(Q)}$ as above we have
\[
\lambda\in\rho(A)\iff
e^\lambda\in\rho(\BB).
\]
Moreover, by  \cite[Prop.~18.7]{BKR:17}
 the operator $A$ generates a strongly continuous semigroup given by
\begin{equation}\label{eq:HG-TE}
\bigl(T(t)f\bigr)(s)=\BB^k f(t+s-k)\quad \text{if }\;t+s\in[k,k+1) \text{ for } k\in\NN_0,
\end{equation}
where $\BB^0:=Id$.
This shows that the Assumptions~\ref{ma-bcs} are satisfied. To proceed we have to compute the associated Dirichlet operator.

\begin{lemma} For $\lambda\in\rho(A)$ the Dirichlet operator $\Ql\in\sL\bigl(\Cm,\LpneCm\bigr)$ is given by
\begin{equation}\label{eq:Dirich-TE}
\Ql=\epsl\otimes \RelBB.
\end{equation}
\end{lemma}

\begin{proof}
By Lemma~\ref{lem-Gre}.(ii) we know that $Q:{\ker(\lambda-\Am)}\to\dX$ is invertible. Moreover, for $d\in\Cm=\dX$ we have
\begin{align*}
Q\bigl(\epsl\otimes \RelBB\,d\bigr)
=\el\cdot \RelBB\,d-\BB\cdot \RelBB\,d=d
\end{align*}
which proves \eqref{eq:Dirich-TE}.
\end{proof}

Next we verify that in this context \eqref{eq:strange1} holds.

\begin{lemma}\label{lem:strange-TE}
Let $\lambda\in\rho(A)$. Then for all $0\le\alpha\le1$
\begin{equation}\label{eq:strange-TE}
\bigl(\ela\cdot T(1-\alpha)\Bl \bigr)(s)=
\begin{cases}
\epsl(1+s)\cdot\RelBB\, b&\text{if }0\le s<\alpha,\\
\epsl(1+s)\cdot\RelBB\, b-\epsl(s)\cdot b&\text{if }\alpha\le s\le1.
\end{cases}
\end{equation}
Hence, \eqref{eq:strange1} is satisfied for
\begin{align*}
M&=b\in\sL\bigl(\Lpne,\LpneCm\bigr),\
(M u)(s)=u(s)\cdot b.
\end{align*}
\end{lemma}
\begin{proof} The claim follows from \eqref{eq:Dirich-TE} and \eqref{eq:HG-TE} by the following simple computation.
\begin{align*}
\bigl(\ela\cdot T(1-\alpha)\Bl \bigr)(s)
&=\ela\cdot\Bigl(T(1-\alpha)\bigl(\epsl\otimes\RelBB\,b\bigr)\Bigr)(s)
\\
&=\ela\cdot
\begin{cases}
\epsl(1-\alpha+s)\cdot\RelBB\, b&\text{if }0\le s<\alpha,\\
\epsl(s-\alpha)\cdot\BB\RelBB\, b&\text{if }\alpha\le s\le1,
\end{cases}
\\
&=
\begin{cases}
\epsl(1+s)\cdot\RelBB\, b&\text{if }0\le s<\alpha,\\
\epsl(1+s)\cdot\RelBB\, b-\epsl(s)\cdot b&\text{if }\alpha\le s\le1.
\end{cases}
\qedhere
\end{align*}
\end{proof}
Thus $B$ is $p$-boundary admissible. Next we compute the appropriate reachability space.
 
\begin{corollary}\label{cor:Reach-TE}
If $t\ge m$ then the exact reachability space of the controlled transport equation~\eqref{eq:TE} is given by
\begin{equation*}
\sRBCt=\sRBC=\Lpne\otimes\lin\l\{b,\BB b,\ldots,\BB^{m-1}b\r\}.
\end{equation*}
\end{corollary}

\begin{proof}
Note that by \eqref{eq:HG-TE} we have $T(1)f=\BB f$. Hence, for $t=m$ the assertion follows immediately from Corollary~\ref{cor:Reach} and Lemma~\ref{lem:strange-TE}. Clearly, $\sRBCt$ increases in time $t\ge0$. However, by the Cayley--Hamilton theorem $\lin\{b,\BB b,\ldots,\BB^{l}b\}=\lin\{b,\BB b,\ldots,\BB^{m-1}b\}$ for all $l\ge m-1$ and the claim follows.
\end{proof}

\begin{remark}\label{rem:C-H}
Let $l\le m$ be the degree of the minimal polynomial of $\BB$. Then the previous proof shows that for all $t\ge l$ we even have
\begin{equation*}
\sRBCt=\sRBC=\Lpne\otimes\lin\l\{b,\BB b,\ldots,\BB^{l-1}b\r\}.
\end{equation*}
\end{remark}

\begin{corollary}\label{cor:TE} 
The following assertions are equivalent.
\begin{enumerate}[(a)]
\item Equation~\eqref{eq:TE} is exactly boundary controllable in time $t\ge m$, i.e.,
$\sRBCt=X$.
\item Equation~\eqref{eq:TE} is maximally controllable in time $t\ge m$, i.e., 
$\sRBCt=\sRBCm\,$.
\item $\lin\l\{b,\BB b,\ldots,\BB^{m-1}b\r\}=\CC^m$.
\end{enumerate}
\end{corollary}

\begin{proof} Note that $\ker(\lambda-A_m)=\epsl\otimes\CC^m$. Since by the Stone--Weierstra\ss{} theorem we have
\[\linq\bigcup_{\lambda>\gb(A)} \{\epsl\}=\Lpne,\]
the maximal reachability space equals
\[
\sRBCm=\Lpne\otimes\CC^m = X
\]
and the assertions follow immediately from Corollary~\ref{cor:Reach-TE}.
\end{proof}

\begin{remark}
The previous result characterizes the \emph{exact} maximal boundary controllability by a one-dimensional control in terms of a Kalman-type condition  which is well-known in control theory. 
\end{remark}

Combining Remark~\ref{rem:C-H} and Corollary~\ref{cor:TE} we furthermore obtain the following

\begin{corollary} Let $l\in\NN$ be the degree of the minimal polynomial of $\BB$. If $l < m$, the transport equation~\eqref{eq:TE} is not 
maximally controllable, i.e., $\sRBC \subsetneq\sRBCm$. 
\end{corollary}

\smallbreak

Finally, we investigate positive controllability and consider
\begin{itemize}
\item the positive cone  $X^+:=\LpneRmp$ in the state space $X$,
\item the positive cone  $U^+:=\RR_+$ in the control space $U$,
\item a positive matrix $\BB\in\rM_m(\RR_+)$,
\item a positive control operator $B:=b\in\RR_+^m$.
\end{itemize}
Then by \eqref{eq:HG-TE}--\eqref{eq:Dirich-TE} the operators $T(t)\in\sL(X)$ for $t\ge0$ and $\Bl\in\sL(U,X)$ for $\lambda>\gb(A)$ are positive.
Thus arguing as above using Proposition \ref{prop:reach-pos} and Corollary~\ref{cor:Reach_pos} we obtain the following.

\begin{corollary}\label{cor:pos-Reach-TE}
The exact positive reachability space of the controlled transport equation~\eqref{eq:TE} is given by
\begin{equation*}
\sRBCp=\Lpnep\otimes\co\bigl\{\BB^{k}b \;:\; k\in\NN_0\bigr\}.
\end{equation*}
Hence, the problem is exactly positive controllable if and only if
\begin{equation*}
\co\bigl\{\BB^{k}b \;:\; k\in\NN_0\bigr\} = \RR_+^m.
\end{equation*}
\end{corollary}

\subsection{Vertex control of flows in networks}

The previous example can be easily adapted to cover a transport problem on a network controlled in a single vertex. More precisely, consider a network consisting of $n$ vertices $\{v_1,\dots ,v_n\}$ and $m$ edges $\{e_1,\dots ,e_m\}$. As shown in \cite[Sec.~18.1]{BKR:17}, its structure can be described by 
either the transposed weighted adjacency matrix  
$\AAA\in\MnC$ given by
\begin{equation*}
 \AAA_{ij}:=\begin{cases}
w _{jk} & \mbox{if $v_j\stackrel{e_k}{\longrightarrow}v_i$},  \\
0 & \text{otherwise,}
\end{cases}
\end{equation*}
or by the transposed weighted adjacency matrix of the line graph 
$\BB\in\MmC$ where
\begin{equation*}
\BB_{ij}:=\begin{cases}
w_{ki} & \mbox{if $\stackrel{e_j}{\longrightarrow}v_k \stackrel{e_i}{\longrightarrow}$}\ ,  \\
0 & \text{otherwise.}
\end{cases}
\end{equation*}
To proceed we also need the transposed weighted outgoing incidence matrix
$\Fwmt=:\Psi\in\MmnC$ defined by
\begin{equation*}
    \Psi_{ij} := 
    \begin{cases}
   w_{ij} & \mbox{if $v_j\stackrel{e_i}{\longrightarrow}$}\ , \\
    0 & \text{otherwise}
    \end{cases}
\end{equation*}
and the corresponding unweighted outgoing 
incidence matrix denoted by $\Fim\in\MnmC$.
For the weights we assume $0\le w_{ij} \le 1$, thus all these matrices are positive. Moreover, we assume that $\Psi$ is column stochastic (i.e., the weights on all the outgoing edges from a given vertex sum up to $1$).
For a detailed account of the various graph matrices we refer to \cite[Sec.~18.1]{BKR:17}. Here we only mention the following relations 
\begin{equation}\label{eq:ABrel}
\Psi\AAA = \BB\Psi,\quad 
\Psi R(\lambda,\AAA) = R(\lambda,\BB)\Psi,\quad 
\text{and}\quad
\Fim\Psi = Id_{\CC^n}
\end{equation}
which we will need in the sequel.

\smallbreak
We then consider a transport equation on the $m$ edges imposing $n$ boundary conditions in the vertices, controlled in a single vertex $v_i$, i.e.,
\begin{alignat}{2}\label{eq:TE-net}
\begin{cases}
\dot x(t,s)=x'(t,s),& s\in[0,1],\ t\ge0,\\
x(t,1)=\BB x(t,0)+u(t)\cdot \Psi v,&t\ge0,\\
x(0,s)=0,& s\in[0,1],
\end{cases}
\end{alignat}
where $x:\RR_+\times[0,1]\to\CC^m$, $u:\RR_+\to\CC$,  and  the vertex $v=v_i$ is represented by the $i$-th canonical basis vector in $\Cn$. To rewrite this equation in an abstract form we take the same state space $X:=\LpneCm$, control space $U:=\CC$  and boundary space $\dX:=\Cm$  as above. Adapting the domain of $\Am$ as 
\[D(\Am):=\l\{f\in\WepneCm: f(1) \in \Rg\Psi \r\}\]
and choosing the control operator $B=b:=\Psi v\in\Cm$  we are in the situation considered in  \cite{EKNS:08} and \cite{BBEAM:13}, see also \cite[Sec.~18.4]{BKR:17}.

Then the \emph{approximate} controllability space  for the network flow problem computed in \cite[Cor.~4.3]{EKNS:08} by our Corollary~\ref{cor:Reach-TE} indeed coincides with the \emph{exact} controllability space.

\begin{corollary}
If $t\ge \min\{m,n\}=:l$ then the exact reachability space of the controlled transport in network problem \eqref{eq:TE-net} equals
\begin{align*}
\sRBCt=\sRBC&=\Lpne\otimes\lin\l\{\Psi v,\BB \Psi v,\ldots,\BB^{l-1}\Psi v\r\}\\
& = \Lpne\otimes \Psi \lin\l\{v,\AAA v,\ldots,\AAA^{l-1} v\r\} .
\end{align*}
\end{corollary}

Note that in big connected networks one usually has $n\le m$, hence the latter space is the more important one for applications.

\smallbreak

Positive control for this problem was already treated in \cite{BBEAM:13} and the approximate positive reachability spaces was computed. However, our approach even yields the \emph{exact} reachability space.

\begin{corollary}
The exact positive reachability space of the controlled transport in network problem \eqref{eq:TE-net} is given by
\begin{align*}
\sRBCp&=\Lpnep\otimes\co\l\{\BB^{k} \Psi v : k\in\NN_0\r\}\\
& = \Lpnep\otimes \Psi \co\l\{\AAA^k v : k\in\NN_0\r\}.
\end{align*}
\end{corollary}

\subsection{Exact Boundary Controllability for Flows in Networks with Dynamical Boundary Conditions}

In this subsection we investigate exact controllability in the situation of \cite[Sect.~3]{EKKNS:10}. Without going much into details we only introduce the necessary facts to state the problem and to compute the exact reachability space $\sRBCt$. 

\smallskip
We start from the transport problem in the network introduced in the previous example, but now change the transmission process in the vertices allowing for dynamical boundary conditions. 
To encode the structure of the underlying network and the imposed boundary conditions we use the incidence matrices introduced above as well as the weighted  incoming incidence matrix $\Fiwp$ given by
\[
\bigl(\Fiwp\bigr)_{ij}:=
  \begin{cases}
   w^+_{ij} & \mbox{if $\stackrel{e_j}{\longrightarrow}v_i$}\ , \\
    0 & \text{otherwise,}
    \end{cases}
\]
for some $0\le w^+_{ij} \le 1$. Defining
\begin{equation}\label{eq:ABrel2}
\AAA := \Phi_w^+\Psi  \quad\text{and}\quad \BB:=\Psi\Phi_w^+
\end{equation}
we obtain the adjacency matrices as above (with different nonzero weights).
We mention that the relations \eqref{eq:ABrel} remain valid also in this case.
\smallbreak

We are then interested in the network transport problem with dynamical boundary conditions in $s=1$ considered already in \cite{Sik:05} and \cite[Sect.~3]{EKKNS:10}, i.e.,
\begin{alignat}{2}\label{eq:TE-net-db}
\begin{cases}
\dot x(t,s)=x'(t,s),& s\in[0,1],\ t\ge0,\\
\dot x(t,1)=\BB x(t,0)+u(t)\cdot \Psi v,&t\ge0,\\
x(0,s)=0,& s\in[0,1],\\
\Phi^- x(1,0) = 0.
\end{cases}
\end{alignat}

To embed this example in our setting we introduce
\begin{itemize}
\item the state space $X:=\LpneCm\times\Cn$ where $1\le p<+\infty$,
\item the boundary space $\dX:=\Cm$,
\item the control space $U:=\CC$,
\item the control operator $B:=\Psi v\in\Cm\simeq\sL(U,\dX)=\sL(\CC,\Cm)$
where $v=v_i$ denotes the $i$-th canonical basis vector of $\Cn$ meaning that the control acts in the $i$-th vertex of the network,
\item the system operator\footnote{By $\delta_s$ we denote the point evaluation in $s\in[0,1]$, i.e., $\delta_s(f)=f(s)$.}
\begin{align*}
\Am:&=\begin{pmatrix}
\diag\bigl(\dds\bigr)_{m\times m}&0\\\Fiwp\delta_0&0
\end{pmatrix}
\quad\text{with domain}\\
D(\Am):&=\l\{\tbinom fd\in\WepneCm\times\Cn:f(1)\in\Rg\Psi\r\},
\end{align*}
\item  the boundary operator $Q:D(\Am)\times\Cn\to\Cm$, $Q\binom fd:=\Fim f(1)-d$,
\item  the operator $A\subset\Am$ with domain $D(A)=\ker Q$.
\end{itemize}

As is shown in \cite[Prop.~3.4]{EKKNS:10} these spaces and operators satisfy all assumptions of Section~\ref{TAF}.
To proceed we first need to compute the associated Dirichlet operator $\Ql$ and an explicit representation of the semigroup operators $T(t)$ for $t\in[0,1]$.

\begin{lemma}
\makeatletter
\hyper@anchor{\@currentHref}%
\makeatother
\label{lem:Diri-SGR-dBC}
\begin{enumerate}[(i)]
\item For each $0\neq \lambda\in\rho(A)$, the Dirichlet operator $Q_{\lambda}\in \sL(\CC^n, {X})$ is given by
\begin{equation}\label{Ql}
   Q_{\lambda} =
   \binom{
   \lambda \ep_{\lambda}\otimes\Psi \RlelAA}
   {\AAA \RlelAA}.
\end{equation}
\item The semigroup $\Ttt$ generated by $A$ is given by\footnote{We use the notations $\bigl[\binom fd\bigr]_1:=f$ and $\bigl[\binom fd\bigr]_2:=d$ for the canonical projections of $\binom fd\in X$.}
\begin{align}
\l[T(t)\tbinom fd\r]_1(s)&=
\begin{cases}\label{eq:[T(t)]1}
f(t+s)&\text{if }\kern22pt 0\le t<1-s,\\
\BB\, V_{t+s-1} f +\Psi d&\text{if }\,1-s\le t\le1,
\end{cases}
\\
\l[T(t)\tbinom fd\r]_2&=\Fiwp\, V_t f +d\kern45pt\text{for }0\le t\le1,\label{eq:[T(t)]2}
\end{align}
where 
\begin{equation}\label{def:V_s}
V_s f:=\int_0^s f(r)\dr\quad \text{for }f\in\LpneCm.
\end{equation} 
\end{enumerate}
\end{lemma}

\begin{proof}
Assertion~(i) is proved in \cite[Prop.~3.8]{EKKNS:10}. Equation~\eqref{eq:[T(t)]2} is shown in the proof of \cite[Prop.~3.4.(iii)]{EKKNS:10}.  The statement \eqref{eq:[T(t)]1} for the first coordinate then follows from \cite[Lem.~6.1]{Sik:05}.
\end{proof}

Next  we apply Proposition~\ref{prop:range-B1} to the present situation.

\begin{lemma}\label{lem:strange-dBC}
Let $\lambda\in\rho(A)$. Then for all $0\le\alpha\le1$
\begin{align}\label{eq:strange-dBC}
\bigl[\ela\cdot T(1-\alpha)\Bl \bigr]_1(s)&=
\begin{cases}
\lambda\epsl(1+s)\cdot\Psi\RlelAA\, v&\text{if }0\le s<\alpha,\\
\lambda\epsl(1+s)\cdot\Psi\RlelAA\, v-\epsl(s)\cdot \Psi v&\text{if }\alpha\le s\le1.
\end{cases}
\\
\bigl[\ela\cdot T(1-\alpha)\Bl \bigr]_2
&=\el\AAA \RlelAA\, v
\end{align}
Hence the equality in \eqref{eq:strange1} is satisfied with
\begin{align*}
M&=\binom{\Psi v}{0}\in\sL\Bigl(\Lpne,\LpneCm\times \Cn\Bigr),\
(M u)(s)=\binom{ u(s)\cdot\Psi v}{0}.
\end{align*}
\end{lemma}
\begin{proof}
Using the explicit representations of $\Ql$ and $T(t)$ given in Lemma~\ref{lem:Diri-SGR-dBC} and the relations \eqref{eq:ABrel} we obtain
\begin{align*}
\bigl[\ela\cdot T(1&-\alpha)\Bl \bigr]_1(s)=\\
&=\ela\cdot
\begin{cases}
\lambda\epsl(1-\alpha+s)\cdot\Psi\RlelAA\, v&\text{if }0\le s<\alpha,\\
\lambda\BB\, V_{s-\alpha}\,\epsl\cdot\Psi\RlelAA\, v+\Psi\AAA \RlelAA v&\text{if }\alpha\le s\le1,
\end{cases}
\\
&=\ela\cdot
\begin{cases}
\lambda\epsl(1-\alpha+s)\cdot\Psi\RlelAA\, v&\text{if }0\le s<\alpha,\\
\bigl(\epsl(s-\alpha)-1\bigr)\cdot\Psi\,\AAA\RlelAA\, v+\Psi\AAA \RlelAA v&\text{if }\alpha\le s\le1,
\end{cases}
\\
&=
\begin{cases}
\lambda\epsl(1+s)\cdot\Psi\RlelAA\, v&\text{if }0\le s<\alpha,\\
\epsl(s)\cdot\Psi\bigl(\lambda\el\RlelAA-Id\bigr)\, v&\text{if }\alpha\le s\le1.
\end{cases}
\\
&=\begin{cases}
\lambda\epsl(1+s)\cdot\Psi\RlelAA\, v&\text{if }0\le s<\alpha,\\
\lambda\epsl(1+s)\cdot\Psi\RlelAA\, v-\epsl(s)\cdot \Psi v&\text{if }\alpha\le s\le1.
\end{cases}
\end{align*}
Similarly, for the second coordinate we have 
\begin{align*}
\bigl[\ela\cdot T(1-\alpha)\Bl \bigr]_2
&=\ela\Bigl(\lambda\Fiwp\,V_{1-\alpha}\,\epsl\cdot\Psi\RlelAA\,v+\AAA \RlelAA\, v\Bigr)\\
&=\ela\Bigl(\bigl(\epsl(1-\alpha)-1\bigr)\cdot\AAA\RlelAA\,v+\AAA \RlelAA\, v\Bigr)\\
&=\el\AAA \RlelAA\, v,
\end{align*}
where we used \eqref{eq:ABrel2}.
\end{proof}

We note that by \cite[Prop.~3.5]{EKKNS:10} the states of the controlled flow at time $t\ge0$ are given by the first coordinate of the states in our ``extended'' state space $X=\LpneCm\times\Cn$. For this reason we also need to compute the first coordinate of $T(1)^k\binom{\Psi g}{0}$.

\begin{lemma}
We have
\begin{equation*}
\l(T(1)\tbinom fd \r)(s)=
\begin{pmatrix}
\Psi\,\Fiwp\, V_s&\Psi\\\kern10pt \Fiwp\,V_1&Id
\end{pmatrix}
\binom{f}{d}
=\binom{\BB\,V_s f+\Psi d}{\Fiwp\, V_1 f +d},
\end{equation*}
where the operator $V_s\in\sL\bigl(\LpneCm,\WepneCm\bigr)$ is defined in \eqref{def:V_s}.
Moreover, for $k\in\NN_1$ we have
\begin{equation}\label{eq:T(k)-dBC}
\l[T(1)^k\tbinom{\Psi g}0\r]_1 (s)
=\Psi(\AAA V_s +\delta_1)^{k-1}\AAA\,V_sg
=(\BB V_s+\delta_1)^{k-1}\BB\Psi\,V_sg.
\end{equation}
\end{lemma}

\begin{proof}
The formula for $T(1)$ follows immediately from Lemma~\ref{lem:Diri-SGR-dBC}.(ii). Since $\Psi\AAA = \BB\Psi$, it suffices to show the second equality in \eqref{eq:T(k)-dBC}. Obviously this equation holds for $k=1$. To verify it for $k>1$ we note that by \eqref{eq:ABrel} the matrix $\Psi$ is left invertible with left inverse $\Fim$. Hence, we obtain
\[
\l[T(1)\tbinom fd\r]_2=\Fim\delta_1 \l[T(1)\tbinom fd\r]_1.
\]
If $\binom fd\in\Rg T(1)$ we can write $f=\Psi h$ and the previous equation implies
\begin{align*}
\l[T(1)\tbinom fd\r]_1 (s)
=\l[T(1)\tbinom {\Psi h}{\delta_1h}\r]_1(s)
=\BB V_s\Psi h+\Psi\delta_1 h
=(\BB V_s +\delta_1)f.
\end{align*}
Now assume that \eqref{eq:T(k)-dBC} holds for some $k\ge1$. Then for $\binom fd=T(1)^k\binom{\Psi g}{0}\in\Rg T(1)$ we conclude
\begin{align*}
\l[T(1)^{k+1}\tbinom{\Psi g}0\r]_1(s)
&=\l[T(1)\cdot T(1)^{k}\tbinom{\Psi g}0\r](s)\\
&=(\BB V_s +\delta_1)\cdot(\BB V_s+\delta_1)^{k-1}\,\BB\Psi\,V_sg\\
&=(\BB V_s+\delta_1)^{k}\,\BB\Psi\,V_s g.\qedhere
\end{align*}

\end{proof}

The previous two lemmas together with Corollary~\ref{cor:Bn} imply the following result.

\begin{corollary}\label{cor:Bn-dBC}
For $l\in\NN_2$ and $u\in\Lpnl$ we have
\begin{align}
\bigl[\sBl u\bigr]_1(s)
&=\Psi\biggl( u_0\otimes v+\sum_{k=1}^{l-1}\bigl(\AAA V_s+\delta_1\bigr)^{k-1}\, V_s(u_{k}\otimes\AAA v)\biggr) \nonumber\\
&=u_0\otimes \Psi v+\sum_{k=1}^{l-1}\bigl(\BB V_s+\delta_1\bigr)^{k-1}\,V_s(u_{k}\otimes\BB\Psi\, v)\label{eq:edBC}
\end{align}
where $u_k\in\Lpne$ is defined as in \eqref{eq:def-uk}.
\end{corollary}

Using this explicit representation of the controllability map we now compute the exact reachability space for the control problem given in \eqref{eq:TE-net-db}.

\begin{corollary}\label{cor:Reach-dBC}
If $t\ge\min\{m,n\}=:l$ then the exact reachability space of the controlled flow with dynamic boundary conditions \eqref{eq:TE-net-db} is given by\footnote{Here we define $\rW^{0,p}[0,1]:=\Lpne$.}
\begin{align*}
\bigl[\sRBCt\bigr]_1
&\subseteq
\l\{
\Psi\sum_{k=0}^{l}\l(u_{k}\otimes\AAA^k\, v\r):u_k\in\Wkpne\text{ for }0\le k\le l
\r\}\\
&=
\l\{
\sum_{k=0}^{l}\l(u_{k}\otimes\BB^k\Psi\, v\r):u_k\in\Wkpne\text{ for }0\le k\le l
\r\}.
\end{align*}
\end{corollary}

\begin{proof}
The equality of the two sets on the right-hand-side follows immediately from \eqref{eq:ABrel}.
To show the inclusion in the second set we combine Corollaries~\ref{cor:Reach} and  \ref{cor:Bn-dBC}.
First observe, that for the operators $\BB$, $V_s$,  and $\delta_1$ we have
\[
\BB V_s f = V_s \BB f, \quad \BB \delta_1 f = \delta_1\BB f,\quad
\delta_1 V_s f = V_1f 
\]
for every $f\in\LpneCm$ while 
\[ \delta_1^k  f = \delta_1f =f(1)\quad\text{for } k\ge 1.\]
So, when expanding $(\BB V_s+\delta_1)^{k-1}V_s$ we can rearrange the terms to obtain expressions of the form
\[\alpha_i \BB^i V_{s_1}\cdots V_{s_{i+1}},\quad 0\le i\le k-1,\]
where $\alpha_i$ are scalar coefficients and $s_j\in\{s,1\}$, $1\le j\le i+1$. 
Next, for arbitrary $u\in\Lpne$ 
and $0\le k\le l$ we have 
\[ V_{s_1}\cdots V_{s_k} u \in \Wkpne, \quad s_j\in\{s,1\}, 1\le j\le k. \]
Combining these facts we obtain the desired result by considering \eqref{eq:edBC} for all $u\in\Lpnl$.
\end{proof}

From the previous Corollary we immediately obtain the following result which improves \cite[Thm.~3.10]{EKKNS:10} and shows that $ \bigl[\saRBCt\bigr]_1$ is constant for $t\ge\min\{m,n\}=:l$.

\begin{corollary}\label{cor:Reach-approx-dBC}
If $t\ge\min\{m,n\}=:l$ then the approximate controllability space of the controlled flow with dynamic boundary conditions \eqref{eq:TE-net-db} is given by
\begin{align*}
\bigl[\saRBCt\bigr]_1
&=\Lpne\otimes\lin\l\{\Psi v,\BB\,\Psi v,\ldots,\BB^{l-1}\Psi v\r\}\\
&=\Lpne\otimes\Psi\lin\l\{v,\AAA v,\ldots,\AAA^{l-1}v\r\}.
\end{align*}
\end{corollary}

\smallskip
In the same manner as before we also obtain the following result on positive controllability.

\begin{corollary}\label{cor:Reach-pos-dBC}
The approximate positive controllability space of the controlled flow with dynamic boundary conditions \eqref{eq:TE-net-db} is given by
\begin{align*}
\bigl[\saRBCp\bigr]_1
&=\Lpne\otimes\coc\l\{\BB^k\Psi v : k\in\NN_0\r\}\\
&=\Lpne\otimes\Psi\,\coc\l\{\AAA^k v: k\in\NN_0\r\}.
\end{align*}
\end{corollary}

\section*{Conclusion}
Using a new characterization of admissible boundary control operators (see Proposition~\ref{prop:range-B1}) we are able to describe explicitly the  \emph{exact} reachability space of the abstract boundary control system  \SBC{}, cf. \eqref{ACPBC}. Moreover, this approach allows us also to determine the \emph{positive} reachability space obtained allowing only positive control functions. Our results generalize and improve the ones obtained in the former works \cite{BBEAM:13, EKNS:08, EKKNS:10} where only approximate controllability or positive controllability under quite restrictive assumptions are studied.

\bigskip
{\small

\noindent
\emph{Klaus-Jochen Engel}, University of L'Aquila, Department of Information Engineering, Computer Science and Mathematics, Via Vetoio, Coppito, I-67100 L'Aquila (AQ), Italy,\\
\texttt{klaus.engel@univaq.it}

\medskip\noindent
\emph{Marjeta Kramar Fijav\v{z}}, University of Ljubljana, Faculty of Civil and Geodetic Engineering, Jamova 2, SI-1000 Ljubljana, Slovenia /
Institute of Mathematics, Physics, and Mechanics,
Jadranska 19, SI-1000 Ljubljana, Slovenia,\\
\texttt{marjeta.kramar@fgg.uni-lj.si}

}
\end{document}